\documentclass[11pt,reqno]{amsart}

\usepackage{amsmath}

\usepackage{amsmath,amsfonts,amssymb}

\usepackage{amsmath,amssymb,amsfonts,mathrsfs,verbatim,enumitem,pstricks,amsthm, graphicx}
\usepackage[all]{xy}
\usepackage{centernot, xr}
\usepackage{hyperref}

\newtheorem{thm}{Theorem}[section]

\newtheorem{lmm}[thm]{Lemma}

\newtheorem{que}[thm]{Question}

\newtheorem*{thm_cnj}{Recursive Formula}
\newtheorem*{mresult_cnj}{Main result (modulo a conjecture)}

\def \lra{\longrightarrow}

\usepackage{hyperref}
\hypersetup{
	colorlinks,
	citecolor=red,
	filecolor=black,
	linkcolor=blue,
	urlcolor=black
}
\def\bge{\begin{equation}}
\def\ede{\end{equation}}                
\def\bgd{\begin{displaymath}}         
\def\edd{\end{displaymath}}            
\def\bgee{\begin{equation*}}           
\def\edee{\end{equation*}}

\usepackage{graphicx}
\usepackage{placeins}

\begin{document}

\title[Enumeration of rational cuspidal curves via the WDVV equation]{Enumeration of rational cuspidal curves via the WDVV equation}

\author[I. Biswas]{Indranil Biswas}

\address{Department of Mathematics, Shiv Nadar University, NH91, Tehsil
Dadri, Greater Noida, Uttar Pradesh 201314, India}

\email{indranil.biswas@snu.edu.in, indranil29@gmail.com}

\author[A. Choudhury]{Apratim Choudhury}

\address{Institut f{\"u}r Mathematik, Humboldt-Universit{\"a}t zu Berlin,
Unter den Linden 6, Berlin- 10099, Germany}

\email{apratim.choudhury@hu-berlin.de}

\author[R. Mukherjee]{Ritwik Mukherjee}
\address{School of Mathematical Sciences, National Institute of Science Education and Research, Bhubaneswar, 
An OCC of Homi Bhabha National Institute, Khurda 752050, Odisha, India.}
\email{ritwikm@niser.ac.in}

\author[A. Paul]{Anantadulal Paul}
\address{International Center for Theoretical Sciences, Survey No. 151, Hesaraghatta, Sivakote, Bangalore, 560089, India.}
\email{anantadulal.paul@icts.res.in}


\date{}

\begin{abstract}
We give a conjectural 
formula for the characteristic number of rational cuspidal curves in $\mathbb{P}^2$ 
by extending the idea of Kontsevich's recursion formula (namely pulling back the 
equality of two divisors in $\overline{M}_{0,4}$).
The key geometric input that is needed here is that 
in the closure of rational 
cuspidal curves, there are two component rational curves which are tangent to each other at 
the nodal point. While this fact is geometrically quite believable, we haven't as yet 
proved it; hence our formula is for the moment conjectural.
The answers that we obtain 
agree with what has been computed earlier 
{by} Ran, 
Pandharipande, Zinger and Ernstr\"{o}m and Kennedy. 
 
We extend this technique (modulo another conjecture) to obtain the characteristic 
number of rational quartics with an $E_6$ singularity.
\end{abstract}

\maketitle
\tableofcontents

\section{Introduction}

The enumerative geometry of rational curves on projective spaces is a rich subject with a long history. It 
has been studied by mathematicians for more than a hundred years. Using the concept of moduli space of stable 
maps and Gromov-Witten Invariants, Kontsevich solved the following question: 
\begin{que}
\label{qu}
How many rational degree $d$ curves are there in $\mathbb{P}^2$ that pass through $3d-1$ generic points? 
\end{que}

Since Kontsevich's elegant recursive formula for the above question, a lot of development has occurred that has 
made an in depth study of moduli space of curves and enumerative geometry. 
Although Gromov-Witten invariants have become an invaluable tool in modern enumerative geometry, actually
only a few tools are available in the literature to explicitly
compute these invariants. Except for favourable cases, computing higher genus Gromov-Witten invariants in general is an extremely
challenging problem in this subject, and it still needs to be solved.

{A collection of} degeneration techniques have appeared to be successful in studying Gromov-Witten invariants. Degenerations suggest that we
must expand the set of curves under consideration while enumerating rational curves through general linear spaces
employing degeneration methods.
Let $H$ be the hyperplane in $\mathbb{P}^n$. The number of
degree $d$ rational curves incident to 
{a collection of} general linear spaces and tangent to $H$,
with various multiplicities along 
{a collection of} general linear subspaces of $H$, are known as
enumerative invariants of $\mathbb{P}^n$. Then, the way degeneration methods apply to study enumerative problems is as follows: one specializes a general linear space to lie entirely in $H$, and hence, one needs to compute the resulting degenerations and multiplicities. Thus, even though one is only interested in ordinary enumerative invariants, one is forced to study all enumerative invariants. This is why questions involving tangencies are fundamental in this subject. For projective space, enumerative invariants are commonly known as relative Gromov-Witten invariants.
Hence, it is natural that problems involving high-order
contact between rational curves appear as sub-problems while understanding enumerative invariants for projective space. The
famous results, amongst others due to Caporaso-Harris \cite{CH} and Gathmann \cite{Gath1} are successful evidences of the technique.

In the spirit of \cite{RVakil_deformation}, one can construct a smooth curve on projective space whose deformation
space may have any given number of components, with a possibly arbitrary number of singularities of given
type, with arbitrarily nonreduced behaviour along various subsets, as the deformation space may be as bad as possible. Due to this, the enumeration of curves with various types of singularities plays an essential part in the modern development of this field. Therefore, the questions invoking rational curves with singularities turn out to be highly challenging, and only a few results are available in the literature. In this context, one can ask the following ambitious problem:

\begin{que}\label{rational_curves_with_singularity}
How many degree $d$ rational curves are there in $\mathbb{CP}^2$ that have a singularity of a certain type, and that pass through the
correct number of generic points?
\end{que}

For some cases, the above question has been extensively studied; however, to the best of our knowledge, there are no general results available for arbitrary type of singularity. When the singularity is an $m$ fold singular point, very recently, in \cite{IRCAA_adv}, we obtained a general recursive formula to enumerate plane rational curves with an $m$ fold singular point satisfying appropriate point constraint, which was only known for $m \leq 3$ earlier.

In this article, we want to understand the question \ref{rational_curves_with_singularity} when the singularity is a cusp. This has been studied among others by Vakil (\cite{Va_Sh}),
Ran (\cite{Ran3}), 
Pandharipande (\cite{Rahul1}), Zinger (\cite{g2p2and3}) and Ernstr\"{o}m and Kennedy (\cite{ken}). 
The question of counting curves with cusps has also been studied more recently 
from the point of view of tropical geometry. 
The paper by Ganor and Shustin (\cite{GS}), studies the question of enumerating genus $g$ curves with certain number of 
nodes and one cusp, using methods from tropical geometry.

Kontsevich's solution to computing the number of rational curves is very elegant. The idea is that we look at 
$\overline{M}_{0,4}$, 
the moduli space of four marked points on a sphere. This space space is isomorphic to $\mathbb{P}^1$. We then realize 
that any two points determine the same divisor. We pull the divisor back on $\overline{M}_{0,4}(\mathbb{P}^2, d)$, 
intersect with appropriate cycles of complementary dimension, which gives us equality of numbers, which in turn gives us 
the famous Kontsevich's recursion formula. The equivalence of the two divisors in $\overline{M}_{0,4}$ is also 
sometimes referred to as the WDVV equation.

The goal of this manuscript is to show that by extending Kontsevich's idea (combined with an extra geometric fact), we 
can enumerate rational cuspidal curves in $\mathbb{P}^2$. It is likely that this idea can be used to enumerate curves with 
higher singularities. The explicit formula is given in section \ref{exp_formula}. In section \ref{ldc} 
we do a few consistency checks and observe that our numbers are consistent with those computed earlier
(including a few numbers that are not directly given in \cite{g2p2and3}, but can be easily obtained by modifying the 
method of Zinger). 

The following is the main result of this paper, modulo a conjecture.

\begin{mresult_cnj}
Let $d,\, m$ be positive integers and $n$ a nonnegative integer. Let $\mathsf{C}^d_m(n)$\footnote{
{Note that when $n=0, 1, 2$ the number $\mathsf{C}^d_m(n)$ gives us the number of rational degree $d$ curves in $\mathbb{P}^2$ with a free cusp (cuspidal point is somewhere in the plane) passing through $3d-2$ points, cusp lying on a fixed line and passing through $3d-3$ points, the cusp being positioned at a fixed point and passing through $3d-4$ points in general position, respectively}.} denote the number of rational curves (in $\mathbb{P}^2$) 
with a cusp lying on the intersection of $n$ generic lines and passing through $m$
generic points, where $n+m\, =\, 3d-2$. Then, using WDVV technique, there is a recursive formula to enumerate all the numbers $\mathsf{C}^d_m(n)$
({see Equation \ref{cusp_main_formula_wdvv}}).
\end{mresult_cnj}
{Observe that there could not be any rational lines and conics with a cusp. Hence, the numbers $\mathsf{C}^d_m(n)$ are nonzero for $d \geq 3$. When $n\geq 3$, we are asking the cuspidal curves to lie on the intersection of three or more lines in $\mathbb{P}^2$, which leads to an empty intersection. Consequently, $\mathsf{C}^d_m(n) = 0$ for $n\geq 3$. These cases will play the role of the base cases for our recursion}.

Denote by $\mathsf{T}_{m_1, m_2}^{d_1, d_2}(n)$ the number of two component rational curves of degree $d_1$ and $d_2$ such that the $d_1$
component passes through $m_1$ 
points, the $d_2$ component passes through $m_2$ points, they are tangent to each other at the nodal point, and the nodal point 
passes through the intersection of $n$ generic lines, where $m_1+m_2+n+1 = 3(d_1+d_2)-2$. The locus of such elements comprises of elements of the following type:
\begin{figure}[hbt!]
\label{pic idea1}
\begin{center}\includegraphics[scale = 0.6]{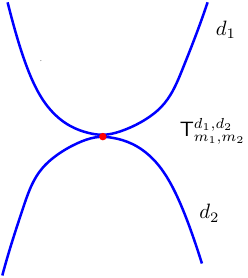}
\end{center}
\end{figure}
\FloatBarrier
These numbers complete our recursion formula for $\mathsf{C}^d_m(n)$. On the other
hand, the formula for $\mathsf{T}_{m_1, m_2}^{d_1, d_2}(n)$ has already been established in \cite{Gath1}, and ours agree with them.

There has been a great deal of earlier work on this problem. In \cite{Rahul1}, Pandharipande investigated the geometry of divisors on $\overline{M}_{0,n}(\mathbb{P}^r, d)$,
and as an application of this method,
he calculated classical tangency characteristic numbers of
rational curves in projective space as well as enumerated $1-$ cuspidal 
rational plane curves passing through $3d - 2$ general points in $\mathbb{P}^2$. A rational plane curve $C$ is $1-$cuspidal if the singularities of $C$ consist of
nodes and exactly $1$ cusp. The method employed is the following:

Let $M_{0,0}(\mathbb{P}^2, d)$ be $\overline{M}_{0,0}(\mathbb{P}^2, d)$ minus the boundary. Let $Z \,\subset\, M_{0,0}(\mathbb{P}^2, d)$ be the subvariety of maps that are not immersions. Then it can be easily seen
that $Z$ is of pure codimension $1$ and the generic element of every component corresponds to a $1$-cuspidal rational plane curve provided
$d \,\geq\, 3$. Let $\mathcal{Z}$ be the Weil divisor obtained by taking
the closure of $Z$ in $\overline{M}_{0,0}(\mathbb{P}^2, d)$. Then in \cite{Rahul1}, it has been shown that 
{the number $C_d$ of irreducible $1$-cuspidal rational plane curves passing through $3d-2$ general points is given by $$C_d\ =\ \mathcal{Z} \cdot \mathcal{H}^{3d-2}$$
on $\overline{M}_{0,0}(\mathbb{P}^2, d)$, where the class $\mathcal{H}$ is determined by passing through point condition}.

In \cite{ken}, Ernstr\"{o}m and Kennedy considered the moduli space of \textit{stable lifts} $\overline{M}^1_{0,n}(\mathbb{P}^2, d)$, parametrizing
those stable maps to the incidence correspondence of
points and lines in $\mathbb{P}^2$, which can only be lifted from maps to $\mathbb{P}^2$, and their degenerations, and they described all possible boundary divisors of this space. As a consequence, from the linear equivalence of special boundary divisors in this space, they obtain several recursive formulas to enumerate the characteristic numbers of rational plane curves of degree $d$, with one cusp with many insertion conditions.

Our method is straightforward. 
{Let the space $\mathcal{S}_{0,4}$ be a subset of 
$M_{0,4}(\mathbb{P}^2, d)$ consisting of rational curves (with a smooth domain) such that it has a cusp at the first marked point. An element of this
space can be visualized as follows:
\begin{figure}[hbt!]
\label{pic idea2}
\begin{center}\includegraphics[scale = 0.75]{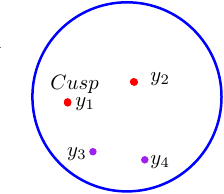}
\end{center}
\end{figure}
\FloatBarrier
Denote by $\overline{\mathcal{S}}_{0,4}(\mathbb{P}^2, d)$ the closure of $\mathcal{S}_{0,4}(\mathbb{P}^2, d)$ inside $\overline{M}_{0,4}(\mathbb{P}^2, d)$. The objects in the closure of $\mathcal{S}_{0,4}(\mathbb{P}^2, d)$ can be visualized by the following picture:
\begin{figure}[hbt!]
\label{pic idea3}
\begin{center}\includegraphics[scale = 0.7]{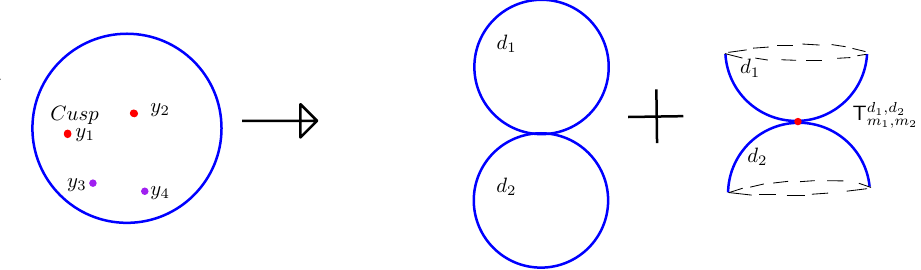}
\end{center}
\end{figure}
\FloatBarrier
Notice that in $\mathsf{T}_{m_1, m_2}^{d_1, d_2}(n)$ component the marking $y_1$ goes to the nodal point}.
We now simply pullback the WDVV equation on $\overline{M}_{0,4}(\mathbb{P}^2, d)$ via the map
$$ \pi\ :\ \overline{M}_{0,4}(\mathbb{P}^2, d)\ \lra\ \overline{M}_{0,4},$$
and intersect it with the cycle $[\overline{\mathcal{S}}_{0,4}(\mathbb{P}^2, d)]$ inside $\overline{M}_{0,4}(\mathbb{P}^2, d)$ having chosen a suitable class of
complementary dimension (see Section \ref{derivation_cupidal_curves}). Therefore, by computing the above intersection together with the geometric fact that in the closure of the space of rational curves with $1$ cusp, there are bubble maps of degree $d_1$ and $d_2$ such that they are tangent to each other at the nodal point, and hence enumerating $\mathsf{T}_{m_1, m_2}^{d_1, d_2}(n)$, we get our recursive formula for $\mathsf{C}^d_m(n)$ for all $d, m,$ and $n$. 

\section{Recursive formula for cuspidal curves}\label{exp_formula}

In this section we present the explicit formula to compute the characteristic number of rational cuspidal curves in 
$\mathbb{P}^2$ that we have obtained via the WDVV equation. Let $\mathsf{N}^{d}_{m}$ 
denote the number of rational degree $d$ curves in $\mathbb{P}^2$ through $m$ points, where $m=3d-1$. 
If $m\,\neq\, 3d-1$, then formally define this number to be zero.
Next, let $\mathsf{C}^d_m(n)$ denote the number of rational curves (in $\mathbb{P}^2$) 
with a cusp lying in the intersection of $n$ generic lines and passing through $m$
generic points, where $n+m = 3d-2$. Again, if $n+m\neq 3d-2$, then formally define this 
number to be zero. Finally, define 
\begin{align*}
\mathsf{T}^{d_1, d_2}_{m_1, m_2}(n) 
\end{align*}
to be the number of two component rational curves of degree $d_1$ and $d_2$ such that the $d_1$ component passes through $m_1$ 
points, the $d_2$ component passes through $m_2$ points, they are tangent to each other at the nodal point and the nodal point 
passes through the intersection of $n$ generic lines, if $m_1+m_2+n+1 \,=\, 3(d_1+d_2)-2$. 
As before, if $m_1+m_2+n+1 \neq 3(d_1+d_2)-2$, then this number is zero. 
We obtain the following recursive formula to enumerate rational cuspidal curves 
in $\mathbb{P}^2$: 

\begin{thm_cnj}
\label{mt_cusp}
Let $\mathsf{N}^{d}_{m}$,\, $\mathsf{C}^d_m(n)$ and $\mathsf{T}^{d_1, d_2}_{m_1, m_2}(n)$ 
be defined as above. Then, 
\begin{align}
\mathsf{C}^d_m(n)& \,=\,0 \qquad \textnormal{if} ~~n \,\geq\, 3 \qquad \forall ~m,\, d \qquad \textnormal{and} \label{base_rel1}\\
\mathsf{C}^1_{m}(n)&\, =\, 0 \qquad \textnormal{and} \qquad \mathsf{C}^2_m(n)\,=\, 0 \qquad \forall ~m,\,n. \label{base_rel2}
\end{align}
Furthermore, when $d\,\geq\, 3$, the following identity holds
$$
\mathsf{C}^d_{3d-2-n}(n)\ =\ d^2\mathsf{C}^d_{3d-3-n}(n+1)-d^2\mathsf{C}^d_{3d-4-n}(n+2)
$$
\begin{equation}\label{cusp_main_formula_wdvv}
+ \sum_{\substack{m_1+m_2=3d-4-n, \\ d_1+d_2=d}}\binom{3d-4-n}{m_1}d_2^2\Big(\mathsf{C}^{d_1}_{m_1}(n)\mathsf{N}^{d_2}_{m_2}d_1^2 d_2
-\mathsf{C}^{d_1}_{m_1+1}(n)\mathsf{N}^{d_2}_{m_2}d_1 d_2+\mathsf{T}^{d_1, d_2}_{m_1, m_2}(n)d_1 
-\mathsf{T}^{d_1, d_2}_{m_1+1, m_2}(n)\Big).
\end{equation}
\end{thm_cnj}

Note that $\mathsf{N}^{d}_{m}$ can be computed via Kontsevich's recursion formula. In Section \ref{T_num} 
a formula is given to compute $\mathsf{T}^{d_1, d_2}_{m_1, m_2}(n)$. Equations \eqref{base_rel1} and \eqref{base_rel2} 
are base cases of the recursion. Using all this information, the main formula \eqref{cusp_main_formula_wdvv} 
yields the characteristic number of rational cuspidal curves in $\mathbb{P}^2$. 

We have written a mathematica program to implement 
the formula given by equation \eqref{cusp_main_formula_wdvv}.
{The implementation of the code can be obtained from the repository}
\[ \textnormal{\url{https://github.com/AnantaSpace/RATIONAL-CUSPIDAL-CURVES-VIA-WDVV}}  \]

\section{Kontsevich's recursion formula} 

For the convenience of the reader, it is recalled how to compute $\mathsf{N}^{d}_{m}$. 
Let $M_{0,n}(\mathbb{P}^2, d)$ denote the moduli space of rational degree $d$ 
curves into $\mathbb{P}^2$ and $\overline{M}_{0,n}(\mathbb{P}^2, d)$ its stable map compactification. This space admits 
$n$ distinct evaluation maps into $\mathbb{P}^2$; the $i^{\textnormal{th}}$ evaluation map 
will be denoted by $\textnormal{ev}_i$. Hence, given a cycle $\mu$ in $\mathbb{P}^2$, there is a corresponding cycle 
$\textnormal{ev}_i^{*}(\mu)$ in $\overline{M}_{0,n}(\mathbb{P}^2, d)$; this denotes the subspace of curves, such that the 
$i^{\textnormal{th}}$ marked point intersects $\mu$. Let $\mathcal{H}$ be the divisor in $\overline{M}_{0,n}(\mathbb{P}^2, d)$ that 
corresponds to the subspace of curves, whose image passes through 
{a fixed} generic point. Now consider the four 
pointed moduli space and the forgetful map
\begin{align*}
\pi\ :\ \overline{M}_{0,4}(\mathbb{P}^2, d)\ &\longrightarrow\ \overline{M}_{0,4}. 
\end{align*}
The above map $\pi$ simply forgets the map to $\mathbb{P}^2$; if the domain of a map to $\mathbb{P}^2$ is unstable, then $\pi$ stabilizes it. 
Now note that $\overline{M}_{0,4}$ is path connected because it is isomorphic to $\mathbb{P}^1$; hence 
any two points determine the same divisors. Consequently, we have
\begin{align*}
(12|34) & = (13|24)
\end{align*}
as divisors in $\overline{M}_{0,4}$. So,
\begin{align}
\pi^*(12|34)\ & =\ \pi^*(13|24) \label{1234_eq_1324}
\end{align}
as cycles in $\overline{M}_{0,4}(\mathbb{P}^2, d)$. Now consider the following cycle
\begin{align*}
[\mathcal{Z}] \ &:=\ \textnormal{ev}_1^*(a^2)\cdot \textnormal{ev}_2^{*}(a^2)\cdot 
\textnormal{ev}_3^*(a)\cdot \textnormal{ev}_4^*(a)\cdot \mathcal{H}^{3d-4},
\end{align*}
where $a$ denotes the class of a line in $\mathbb{P}^2$. 
By intersecting the left hand side and right hand side of equation \eqref{1234_eq_1324}, 
we obtain an 
equality of numbers, namely 
\begin{align}
[\pi^*(12|34)]\cdot [\mathcal{Z}] & = [\pi^*(13|24)]\cdot [\mathcal{Z}]. \label{1234_eq_1324_number}
\end{align}
{This phenomena can be represented by the following picture:}
 \begin{figure}[hbt!]
\label{pic idea4}
\begin{center}\includegraphics[scale = 0.73]{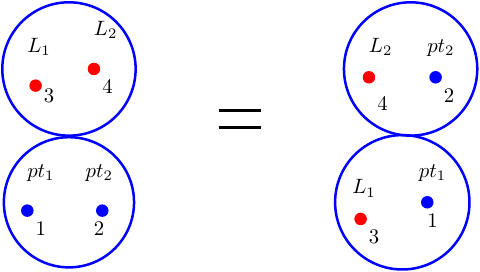}
\end{center}
\end{figure}
\FloatBarrier

We will now explain how to compute the left hand side of equation \eqref{1234_eq_1324_number}. 
First label the four ordered marked points (on the domain) as $y_1,\, y_2,\, y_3$ and $y_4$. 
Hence, the space $\pi^{-1}(12|34)$ comprises of bubble maps of type $(d_1,\, d_2)$. with $d_1+d_2\,=\,d$, 
where $y_1,\, y_2$ are on the $d_1$ component, while $y_3$ and $y_4$ are on the $d_2$ component. 
First consider the case where $d_1\,=\,0$ and $d_2\, =\,d$. This means that the $d_1$ component is constant. 
Therefore, it implies that its intersection with $\mathcal{Z}$ is zero, because we require $y_1$ and $y_2$ 
to go to two different general points. Hence, this component does not show up after intersecting with $\mathcal{Z}$. 
Next, consider the case where $d_1\,=\,d$ and $d_2\,=\,0$. Here, $y_3$ and $y_4$ are required to pass through 
two different generic lines; that means that the entire component gets mapped to a generic point (the intersection of 
the two generic lines). Furthermore, $y_1$ and $y_2$ get mapped to two generic points, and the entire image 
passes through $3d-4$ points. Therefore, the entire image passes through $3d-1$ points. Hence, what we get is 
\begin{align*}
\mathsf{N}^d_{3d-1}. 
\end{align*}
Consider the case where $d_1,\, d_2\,>\,0$. We will compute its intersection with $\mathcal{Z}$. 
Let $m_1+m_2 \,=\, 3d-4$, where $m_1,\, m_2 \,\geq\, 0$. 
Out of the $3d-4$ points, choose $m_2$ points and place the degree $d_2$ 
curve through them. Through the remaining $m_1$ points place a degree $d_1$ curve; this will actually 
pass through $m_1+2$ points, since $y_1$ and $y_2$ get mapped to two generic points. There are 
$d_1d_2$ choices for the nodal point. Finally, intersecting with $\textnormal{ev}_3^*(a)$ and 
$\textnormal{ev}_4^*(a)$ yields a factor of $d_2^2$. Therefore, what we get after intersecting with $\mathcal{Z}$ is 
\begin{align*}
\sum_{\substack{m_1+m_2=3d-4}}\binom{3d-4}{m_1} \mathsf{N}^{d_1}_{m_1+2} \mathsf{N}^{d_2}_{m_2} d_1d_2^3. 
\end{align*}
Hence, it follows that 
\begin{align}
[\pi^*(12|34)]\cdot [\mathcal{Z}]\ & =\ \mathsf{N}^d_{3d-1} + 
\sum_{\substack{m_1+m_2=3d-4, \\ d_1+d_2=d}}\binom{3d-4}{m_1} \mathsf{N}^{d_1}_{m_1+2} \mathsf{N}^{d_2}_{m_2} d_1d_2^3. \label{kr1}
\end{align}
Next, the right-hand side of equation \eqref{1234_eq_1324_number} will be computed. 
The space $\pi^{-1}(13|24)$ comprises of bubble maps of type $(d_1,\, d_2)$, with $d_1+d_2\,=\,d$, 
where $y_1,\, y_3$ are on the $d_1$ component, while $y_2$ and $y_4$ are on the $d_2$ component.
First consider the case where $d_1\,=\,0$ and $d_2 \,=\,d$. The intersection with $\mathcal{Z}$ will be 
zero, because $y_1$ gets mapped to a point and $y_3$ to a line and the map is a constant on that component. 
Similarly, the $d_1\,=\,d$ and $d_2 \,=\,0$ configuration will also produce zero after intersecting with $\mathcal{Z}$. 
Now consider the case where $d_1,\, d_2\,>\,0$. 
We will compute its intersection with $\mathcal{Z}$. 
Let $m_1+m_2 \,= \,3d-4$, where $m_1,\, m_2 \,\geq\, 0$. 
Out of the $3d-4$ points, choose $m_1$ points and place the degree $d_1$ 
curve through them. Through the remaining $m_2$ points place a degree $d_2$ curve. These 
two curves will actually be passing through $m_1+1$ and $m_2+1$ points, because $y_1$ and $y_2$ 
are mapped to points. There are 
$d_1d_2$ choices for the nodal point. Finally, intersecting with $\textnormal{ev}_3^*(a)$ and 
$\textnormal{ev}_4^*(a)$ give us a factor of $d_1 d_2$. 
Hence, what we get after intersecting with $\mathcal{Z}$ is 
\begin{align*}
\sum_{\substack{m_1+m_2=3d-4}}\binom{3d-4}{m_1} \mathsf{N}^{d_1}_{m_1+1} \mathsf{N}^{d_2}_{m_2+1} d_1^2d_2^2. 
\end{align*}
Therefore, it follows that 
\begin{align}
[\pi^*(13|24)]\cdot [\mathcal{Z}]\ & =\
\sum_{\substack{m_1+m_2=3d-4, \\ d_1+d_2=d}}\binom{3d-4}{m_1} \mathsf{N}^{d_1}_{m_1+1} \mathsf{N}^{d_2}_{m_2+1} d_1^2d_2^2. 
\label{kr2}
\end{align}
Equating the right hand sides of equations \eqref{kr1} and \eqref{kr2}, it follows that 
\begin{align*}
\mathsf{N}^d_{3d-1}\ & =\ \sum_{\substack{m_1+m_2=3d-4, \\ d_1+d_2=d}}\binom{3d-4}{m_1} \Big(
\mathsf{N}^{d_1}_{m_1+1} \mathsf{N}^{d_2}_{m_2+1} d_1^2d_2^2-\mathsf{N}^{d_1}_{m_1+2} \mathsf{N}^{d_2}_{m_2} d_1d_2^3\Big).
\end{align*}
Combined with the knowledge that $\mathsf{N}^1_{2}\,=\,1$, what the above formula gives is the characteristic 
number of rational degree $d$ curves in $\mathbb{P}^2$. This is the famous Kontsevich's recursion formula. 
The goal of our paper is to explain how this idea can be extended to compute the characteristic number of 
rational cuspidal curves.

\section{Derivation of recursive formula for cuspidal curves}
\label{derivation_cupidal_curves}

It will be explained how the recursive formula is obtained. 
As before, we will do intersection theory on $\overline{M}_{0,4}(\mathbb{P}^2, d)$. 
Let $\mathcal{S}_{0,4}(\mathbb{P}^2, d)$ denote the following subset of 
$M_{0,4}(\mathbb{P}^2, d)$: it is the space of all
rational curves (with a smooth domain) such that it has a cusp at the first marked point. Its closure
inside $\overline{M}_{0,4}(\mathbb{P}^2, d)$ will be denoted by $\overline{\mathcal{S}}_{0,4}
(\mathbb{P}^2, d)$. Consider the forgetful map 
\begin{align*}
\pi\ : \ \overline{M}_{0,4}(\mathbb{P}^2, d)\ \longrightarrow\ \overline{M}_{0,4}. 
\end{align*}
The reader may be reminded that this is not quite the forgetful map; this forgets the map and if the resulting domain is 
unstable, the map $\pi$ stabilizes it. Define the cycle 
\begin{align*}
\mathcal{Z}\ :=\ \textnormal{ev}_1^*(a^n) \textnormal{ev}_2^*(a^2)\textnormal{ev}_3^*(a)\textnormal{ev}_4^*(a) \mathcal{H}^{3d-4-n}. 
\end{align*}
This cycle, when intersected with 
\[[\pi^{*}(12|34)]\cdot [\overline{\mathcal{S}}_{0,4}(\mathbb{P}^2, d)]\]
in $\overline{M}_{0,4}(\mathbb{P}^2, d)$, produces a number. 
Hence, the goal is to identify all the components of the cycle 
\[[\overline{\mathcal{S}}_{0,4}(\mathbb{P}^2, d)]\cdot [\pi^*(12|34)],\] 
which will make it possible to compute its 
intersection with $\mathcal{Z}$. 
The space $[\overline{\mathcal{S}}_{0,4}(\mathbb{P}^2, d)]\cdot [\pi^*(12|34)]$ will comprise of 
components of degree $d_1,\, d_2$, where $d_1+d_2 \,=\, d$. 
As before, label the marked points as 
$y_1$, $y_2$, $y_3$ and $y_4$. 
The points $y_1$ and $y_2$ lie on the $d_1$ component, while the 
points $y_3$ and $y_4$ lie on the $d_2$ component.

First consider the component of 
$[\overline{\mathcal{S}}_{0,4}(\mathbb{P}^2, d)]\cdot [\pi^*(12|34)]$ that corresponds to the 
breakup $d_1 = 0$ and $d_2 =d$. This happens when $y_1$ and $y_2$ come together in 
$\mathcal{S}_{0,4}(\mathbb{P}^2, d)$. The resulting object is a cuspidal degree $d$ curve 
with the marked points $y_3$ and $y_4$ on them and a constant component bubble attached 
at the cuspidal point, with the marked points $y_1$ and $y_2$ on this component. The intersection of 
this component with $\mathcal{Z}$ will produce
\begin{align*}
d^2\mathsf{C}^{d}_{3d-4-n}(n+2). 
\end{align*}
{To see why this is the case}, 
note that after intersection with $\mathcal{Z}$, the cusp will have to lie on the cycle $a^{n+2}$ 
(intersection of $n+2$ generic lines). Furthermore, intersection with $\textnormal{ev}_3^{*}(a)$ and $\textnormal{ev}_4^{*}(a)$ 
multiply by a factor of $d$ each time.

Next, consider the component of 
$[\overline{\mathcal{S}}_{0,4}(\mathbb{P}^2, d)]\cdot [\pi^*(12|34)]$ that corresponds to the 
breakup where $d_1 \,=\, d$ and $d_2 \,=\,0$. This happens when $y_3$ and $y_4$ come together in 
$\mathcal{S}_{0,4}(\mathbb{P}^2, d)$. The resulting object is a cuspidal degree $d$ curve 
with the marked points $y_1$ and $y_2$ on them (with $y_1$ the 
cuspidal point) and a constant component bubble attached 
with the marked points $y_3$ and $y_4$ on that bubble. 
The intersection of 
this component with $\mathcal{Z}$ will give us 
\begin{align*}
\mathsf{C}^{d}_{3d-2-n}(n). 
\end{align*}
{To see this}, note that the $d_2$ component (which is constant) will have to pass through 
the intersection of two lines (due to intersection with $\textnormal{ev}_3^{*}(a)$ and $\textnormal{ev}_4^{*}(a)$). 
The $y_2$ point also maps to a point (due to intersection with $\textnormal{ev}_2^{*}(a^2)$). 
Hence, the total configuration passes through $3d-2-n$ points. Also note that the degree $d_1$ component 
(which is of degree $d$) has a cusp at the point $y_1$ and the cusp lies at the intersection of $n$ lines
(due to intersection with $\textnormal{ev}_2^{*}(a^n)$). That precisely produces 
$\mathsf{C}^{d}_{3d-2-n}(n)$.

Now consider the component of 
$[\overline{\mathcal{S}}_{0,4}(\mathbb{P}^2, d)]\cdot [\pi^*(12|34)]$ that corresponds to the 
breakup $d_1\,>\,0$ and $d_2\,>\,0$. There will be two types of components here. The first type of 
component will certainly be a degree $d_1$ curve and a degree $d_2$ curve, with the degree $d_1$ 
curve having a cusp at $y_1$. Furthermore, the point $y_2$ is also on the degree $d_1$ component 
and the points $y_3$ and $y_4$ are on the degree $d_2$ component.

Intersecting this with $\mathcal{Z}$ gives us 
\begin{align*}
\sum_{m_1+m_2=3d-4-n} \binom{3d-4-n}{m_1}\mathsf{C}^{d_1}_{m_1+1}(n) \mathsf{N}^{d_2}_{m_2}d_1 d_2^3. 
\end{align*}
{To see this}, we do the following: choose $m_1$ points out of the available $3d-4-n$ points and 
place the degree $d_1$ curve through it. This curve will actually be passing through $m_1+1$ points, because 
we are intersecting with $\textnormal{ev}_2^*(a^2)$. Through the remaining $m_2$ points, we place the degree 
$d_2$ curve. There is a factor of $d_1 d_2$ to account for the nodal point. There is also a further factor of $d_2^2$ 
because we are intersecting with $\textnormal{ev}_3^*(a)\textnormal{ev}_4^*(a)$.

Now we come to the main non trivial geometric input that we use; this assertion is 
for the moment a conjecture. 
There is a second type of component.
This comprises a degree $d_1$ curve, a ghost bubble, and a degree $d_2$ curve attached to the ghost bubble. 
The point $y_2$ lies on the degree $d_1$ component. The points $y_3$ and $y_4$ lie on the degree $d_2$ 
component. The point $y_1$ lies on the ghost bubble. Furthermore, the images of the degree $d_1$ and $d_2$ 
are tangent at the image of $y_1$. Basically these are two component curves, intersecting tangentially at the 
nodal point. 
This degeneration can be observed by the following picture:
\begin{figure}[hbt!]
\label{pic idea5}
\begin{center}\includegraphics[scale = 0.8]{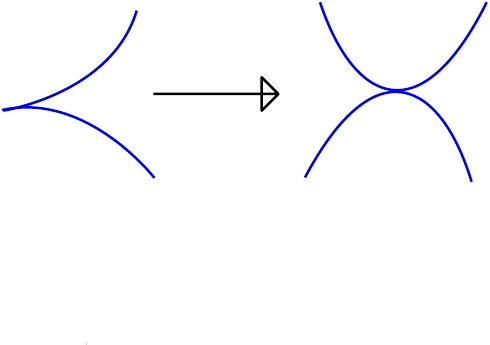}
\end{center}
\end{figure}
\FloatBarrier

Now note that 
intersecting this component with $\mathcal{Z}$ gives us
\begin{align*}
\sum_{m_1+m_2=3d-4-n} \binom{3d-4-n}{m_1} \mathsf{T}^{d_1, d_2}_{m_1+1, m_2}(n) d_2^2. 
\end{align*}
{To see this}, we place the degree $d_1$ curve through $m_1$ points. This will actually pass through
$m_1+1$ points because 
we are intersecting with $\textnormal{ev}_2^*(a^2)$. We also place the degree $d_2$ curve through $m_2$ 
points. There is no factor of $d_1 d_2$ in this case, since only one point of intersection is being counted for 
the nodal point (in the remaining points, they won't be intersecting tangentially). The factor of $d_2^2$ is there 
because we are intersecting with $\textnormal{ev}_3^*(a)\textnormal{ev}_4^*(a)$.

Hence, we conclude that 
\begin{align}
[\overline{\mathcal{S}}_{0,4}(\mathbb{P}^2, d)]\cdot [\pi^*(12|34)]\cdot \mathcal{Z} & = 
\mathsf{C}^{d}_{3d-2-n}(n) + d^2\mathsf{C}^{d}_{3d-4-n}(n+2)\nonumber \\ 
& + \sum_{\substack{d_1+d_2=d, \\m_1+m_2 = 3d-4-n}} \binom{3d-4-n}{m_1} 
\Big( \mathsf{C}^{d_1}_{m_1+1}(n) \mathsf{N}^{d_2}_{m_2}d_1 d_2^3 + \mathsf{T}^{d_1, d_2}_{m_1+1, m_2}(n) d_2^2\Big). \label{12_34_cusp}
\end{align}
To analyze the cycle $[\overline{\mathcal{S}}_{0,4}(\mathbb{P}^2, d)]\cdot [\pi^*(13|24)]$, note that this 
comprises of two component curves of degree $d_1$ and $d_2$, where the points $y_1$ and $y_3$ lie on the 
$d_1$ component, while the points $y_2$ and $y_4$ lie on the degree $d_2$ component.

Now start with the breakup $d_1\,=\,0$ and $d_2\,=\,d$. 
This happens when $y_1$ and $y_3$ come together in 
$\mathcal{S}_{0,4}(\mathbb{P}^2, d)$. The resulting object is a cuspidal degree $d$ curve 
with the marked points $y_2$ and $y_4$ on them and a constant component bubble attached 
at the cuspidal point, with the marked points $y_1$ and $y_3$ on this component. The intersection of 
this component with $\mathcal{Z}$ will produce
\begin{align*}
d\mathsf{C}^{d}_{3d-3-n}(n+1). 
\end{align*}
{Indeed after} intersection with $\mathcal{Z}$, the cusp will have to lie on the cycle $a^{n+1}$ 
(intersection of $n+1$ lines); this is because we are intersecting with $\textnormal{ev}_1^*(a^n) \textnormal{ev}_3^*(a)$.
Next, note that intersection with $\textnormal{ev}_2^{*}(a^2)$ makes the curve pass through $3d-3-n$ points. 
Finally, intersecting with $\textnormal{ev}_4^{*}(a)$, multiplies the by a factor of $d$.

Next, consider the breakup $d_1=d$ and $d_2=0$. 
This happens when $y_2$ and $y_4$ come together in 
$\mathcal{S}_{0,4}(\mathbb{P}^2, d)$. The resulting object is a cuspidal degree $d$ curve 
with the marked points $y_1$ and $y_3$ on them with $y_1$ being the cuspidal point. 
Furthermore, a ghost bubble is attached to this map and the points $y_2$ and $y_4$ lie 
on this component. 
The intersection of 
this component with $\mathcal{Z}$ will give us zero. This is because we are intersecting with 
$\textnormal{ev}_2^{*}(a^2)\textnormal{ev}_4^{*}(a)$; this means the constant component has 
to map to the intersecting of three generic lines, which is empty.

Now consider the component of 
$[\overline{\mathcal{S}}_{0,4}(\mathbb{P}^2, d)]\cdot [\pi^*(13|24)]$ that corresponds to the 
breakup $d_1>0$ and $d_2>0$. As before, there will be two types of components here. The first type of 
component will be a degree $d_1$ curve and a degree $d_2$ curve, with the degree $d_1$ 
curve having a cusp at $y_1$. Furthermore, the point $y_3$ is also on the degree $d_1$ component 
and the points $y_2$ and $y_4$ are on the degree $d_2$ component.

Intersecting this with $\mathcal{Z}$ gives us
\begin{align*}
\sum_{m_1+m_2=3d-4-n} \binom{3d-4-n}{m_1}\mathsf{C}^{d_1}_{m_1}(n) \mathsf{N}^{d_2}_{m_2+1}d_1^2 d_2^2. 
\end{align*}
This is similar to the earlier reasoning.

Next, as before there is a second type of component that is 
there. This comprises of a degree $d_1$ curve, a ghost bubble and a degree $d_2$ curve attached to the ghost bubble. 
The point $y_3$ lies on the degree $d_1$ component. The points $y_2$ and $y_4$ lie on the degree $d_2$ 
component. The point $y_1$ lies on the ghost bubble. Furthermore, the images of the degree $d_1$ and $d_2$ 
are tangent at the image of $y_1$. Basically these are two component curves, intersecting tangentially at the 
nodal point. Now note that 
intersecting this component with $\mathcal{Z}$ gives us 
\begin{align*}
\sum_{m_1+m_2=3d-4-n} \binom{3d-4-n}{m_1} \mathsf{T}^{d_1, d_2}_{m_1, m_2+1}(n)d_1 d_2.
\end{align*}
The factor of $d_1d_2$ is there because we intersect with $\textnormal{ev}_3^*(a)\textnormal{ev}_4^*(a)$ 
(as before there is no factor of $d_1 d_2$ for the bubble point, since only one of them is tangential).

Hence, we conclude that 
\begin{align}
\begin{split}
[\overline{\mathcal{S}}_{0,4}(\mathbb{P}^2, d)]\cdot [\pi^*(13|24)]\cdot \mathcal{Z} & = d\mathsf{C}^{d}_{3d-3-n}(n+1) \\ 
& + \sum_{\substack{d_1+d_2=d, \\m_1+m_2 = 3d-4-n}} \binom{3d-4-n}{m_1} 
\Big(\mathsf{C}^{d_1}_{m_1}(n) \mathsf{N}^{d_2}_{m_2+1}d_1^2 d_2^2+ \mathsf{T}^{d_1, d_2}_{m_1, m_2+1}(n)d_1 d_2\Big). 
\label{13_24_cusp}
\end{split}
\end{align}
Equating the right hand sides of equations \eqref{12_34_cusp} and \eqref{13_24_cusp}, we get the recursion formula 
\eqref{cusp_main_formula_wdvv} in Theorem \ref{mt_cusp}. 

\section{Intersection of Tautological Classes} 
\label{itc_wdvv}

It remains to compute the number $\mathsf{T}^{d_1, d_2}_{m_1, m_2}(n)$. Before it is explained how to 
compute that, we will recapitulate some basic facts about the intersection of tautological classes. 
First consider the moduli space of curves with one marked point, namely 
\begin{align*}
\overline{M}_{0,1}(\mathbb{P}^2, d). 
\end{align*}
On top of this space, there is the universal tangent bundle, given by 
\begin{align*}
\mathbb{L}\ \longrightarrow\ \overline{M}_{0,1}(\mathbb{P}^2, d). 
\end{align*}
Define 
\begin{align*}
\Phi_d(k,j,m)\ :=\ \Big\langle c_1(\mathbb{L}^*)^k\cdot \textnormal{ev}^*(a^j) \cdot \mathcal{H}^m, 
~\overline{M}_{0,1}(\mathbb{P}^2, d) \Big \rangle\ \in\ {\mathbb Z}.
\end{align*}
The above number $\Phi_d(k,j,m)$ is declared to be zero, unless $k+j+m\,=\,3d$.

To compute $\mathsf{T}^{d_1, d_2}_{m_1, m_2}(n)$, it will be necessary to 
know how to compute $\Phi_d(0,j,m)$
and $\Phi_d(1,j,m)$. Before the formula is given, 
we will also need to consider the two marked moduli space, namely 
\begin{align*}
\overline{M}_{0,2}(\mathbb{P}^2, d). 
\end{align*}
Denote by $\mathbb{L}$ the universal tangent bundle on $\overline{M}_{0,2}(\mathbb{P}^2, d)$ over the 
first marked point. Define 
\begin{align*}
{\Phi_d^{(2)}(k,j_1, j_2,m)\ :=\ \Big\langle c_1(\mathbb{L}^*)^k\cdot \textnormal{ev}_1^*(a^{j_1}) \cdot \textnormal{ev}_2^*(a^{j_2}) \cdot \mathcal{H}^m, 
~\overline{M}_{0,2}(\mathbb{P}^2, d) \Big \rangle}.
\end{align*}
Now note that 
\begin{align}
\Phi_{d}(0, j, m) & =~ 0 \qquad \hspace*{2cm}\textnormal{if} ~~ j=0, \nonumber \\
& =~ d~ \mathsf{N}_{3d-1}^{d}, \qquad \hspace*{0.8cm}\textnormal{if} ~~j=1 \qquad \textnormal{and} ~~m = 3d-1, \nonumber \\ 
& =~ 0 \qquad \hspace*{2cm}\textnormal{if} ~~j =1 \qquad \textnormal{and} ~~m \neq 3d-1, \nonumber \\ 
 & =~ \mathsf{N}_{3d-1}^{d} \qquad \hspace*{1.25cm}\textnormal{if} ~~j=2 \qquad \textnormal{and} ~~m = 3d-2, \nonumber \\ 
& =~ 0 \qquad \hspace*{2cm}\textnormal{if} ~~j=2 \qquad \textnormal{and} ~~m \neq 3d-2, \nonumber \\ 
& =~ 0 \qquad \hspace*{2cm}\textnormal{if} ~~j\geq 3. \label{Phi_0_rec_base}
\end{align}
Also note that 
\begin{align}
\Phi_d^{(2)}(0,j_1, 1,m)&= d~ \Phi_d(0,j_1,m) \qquad \textnormal{and} \qquad
{\Phi_d^{(2)}(0,j_1, 2,m) =  \Phi_d(0,j_1,m+1)}. \label{phi_2_to_1_conv}
\end{align}
We are now ready to give the formula for $\Phi_d(1,j,m)$.

\begin{lmm}
\label{itc_c1_pow_k}
The intersection numbers $\Phi_{d}(1, \alpha, m)$ are given by the recursive formula 
\begin{align}
\Phi_{d}(1, j, m) &\ =\ \frac{1}{d^2} \Phi_{d}(0, j, m+1)
-\frac{2}{d} \Phi_{d}(0, j+1, m) \nonumber \\ 
+& \sum_{\substack{m_1+ m_2 = m, \\ d_1 + d_2 = d, ~~d_1, d_2 \neq 0, \\ 
\mu, \nu = 0 ~\textnormal{to} ~2}} \binom{m}{m_1} g^{\mu \nu}\frac{d_2^2}{d^2} 
\Phi_{d_1}^{(2)}(0, j, \mu, m_1) 
\Phi_{d_2}(0, \nu, m_2). \label{phi_k_algo}
\end{align}
\end{lmm}

\begin{proof}
First note that on 
$\overline{M}_{0,1}(\mathbb{P}^2, d)$, the following 
equality of divisors holds:
\begin{equation}
c_1(\mathbb{L}^{*}) \,=\, \frac{1}{d^2}
\Big( \mathcal{H} -2 \textnormal{ev}^*(a) +
\sum_{\substack{d_1+ d_2= d, \\ d_1, d_2 \neq 0}}
d_2^2 \mathcal{B}_{d_1, d_2} \Big)\, ,\label{chern_class_divisor}
\end{equation}
where 
$\mathcal{B}_{d_1, d_2}$ denotes the
boundary stratum corresponding to the splitting into a
degree $d_1$ curve and degree $d_2$ curve with the last marked point
lying on the degree $d_1$ component. This is proved in \cite[Lemma 2.3]{Ionel_genus_one} and 
also in 
\cite[Lemma 3.1]{BMT1}. Now note that the first two terms in the right hand side of equation \eqref{phi_k_algo} 
follow immediately from equation \eqref{chern_class_divisor}; just multiply both sides of the equation with 
$\textnormal{ev}^*(a^j)\mathcal{H}^{m}$ and then evaluate it on $[\overline{{M}}_{0,1}(\mathbb{P}^2, d)]$. 

We will show that 
\begin{align}
\Big\langle \textnormal{ev}^*(a^j)\mathcal{H}^m\mathcal{B}_{d_1, d_2}, ~\overline{{M}}_{0,1}(\mathbb{P}^2, d)
\Big\rangle &=
\sum_{\substack{m_1+ m_2 = m,\\ d_1 + d_2 = d, ~d_1, d_2 \neq 0, \\ 
                             \mu, \nu = 0 ~\textnormal{to} ~2}} \binom{m}{m_1} g^{\mu \nu} 
                          \Phi_{d_1}^{(2)}(0, j, \mu, m_1) 
                                                  \Phi_{d_2}(0, \nu, m_2). \label{ci_k_bd} 
\end{align}
Equation \eqref{ci_k_bd} immediately gives us the last term of equation \eqref{phi_k_algo} 
(the term inside the summation). This gives us equation \eqref{phi_k_algo}.

In order to prove equation \eqref{ci_k_bd}, we do the following. 
Consider the spaces 
\begin{align*}
\overline{{M}}_{0,2}(\mathbb{P}^2, d_1)& \qquad \textnormal{and} \qquad \overline{{M}}_{0,1}(\mathbb{P}^2, d_2).
\end{align*}
Label the two points on $\overline{{M}}_{0,2}(\mathbb{P}^2, d_1)$ as $y_1$ and $y_2$ and the point on 
$\overline{{M}}_{0,1}(\mathbb{P}^2, d_2)$ as $y_3$. Now consider the map 
\[
\textnormal{ev}_2\times \textnormal{ev}_{3}\,:\, \overline{{M}}_{0,2}(\mathbb{P}^2, d_1) \times 
\overline{{M}}_{0,1}(\mathbb{P}^2, d_2) \,\longrightarrow\, \mathbb{P}^2\times \mathbb{P}^2,
\]
where $\textnormal{ev}_i$ denotes evaluation at $y_i$. 
If $\Delta_{\mathbb{P}^2}$ denotes the diagonal in $\mathbb{P}^2\times \mathbb{P}^2$, then note that 
\begin{align}
\mathcal{B}_{d_1, d_2}\ \approx\ 
(\textnormal{ev}_2\times \textnormal{ev}_3)^{-1}(\Delta_{\mathbb{P}^2})\ \subset\ 
\overline{{M}}_{0,2}(\mathbb{P}^2, d_1)\times \overline{{M}}_{0,1}(\mathbb{P}^2, d_2).\label{b_delta}
\end{align}
Let $\mathcal{H}_1$ and $\mathcal{H}_2$ correspond to the divisors in $\overline{{M}}_{0,2}(\mathbb{P}^2, d_1)$ 
and $\overline{{M}}_{0,1}(\mathbb{P}^2, d_2)$, where the image of the curve passes through a generic point. Hence, 
intersecting $\mathcal{B}_{d_1, d_2}$ with $\mathcal{H}^{m_1}_1\mathcal{H}^{m_2}_2$ corresponds to the 
degree $d_1$ curve passing through $m_1$ points and the degree $d_2$ curve passing through $m_2$ points. Therefore,
we conclude that 
\begin{align}
[\mathcal{B}_{d_1, d_2}] \textnormal{ev}^*(a^j)\cdot \mathcal{H}^m\ = \
\sum_{m_1+m_2=m} \binom{m}{m_1}[\mathcal{B}_{d_1, d_2}]\textnormal{ev}^*(a^j) \mathcal{H}_1^{m_1}\mathcal{H}_2^{m_2}.
\label{b_delta2}
\end{align}
Equation \eqref{b_delta}, combined with \eqref{b_delta2} gives us \eqref{ci_k_bd} (using the expression for the diagonal 
of $\mathbb{P}^2$).
\end{proof}

\section{Computation of two component tangential curves}
\label{T_num}
We are now ready to compute $\mathsf{T}^{d_1, d_2}_{m_1, m_2}(n)$, 
the number of two component rational curves of degree $d_1$ and $d_2$ such that the $d_1$ component passes through $m_1$ 
points, the $d_2$ component passes through $m_2$ points, they are tangent to each other at the nodal point and the nodal point 
passes through the intersection of $n$ generic lines, if $m_1+m_2+n+1 = 3(d_1+d_2)-2$. For instance, when $d_1 = 1$, i.e., one of the components is a line, this locus is very classical. The following classical picture can represent an element in this case:
\begin{figure}[hbt!]
\label{pic idea6}
\begin{center}\includegraphics[scale = 0.7]{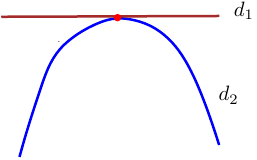}
\end{center}
\end{figure}
\FloatBarrier 
Recall that if $m_1+m_2+n+1 \,\neq\, 3(d_1+d_2)-2$, then the above number is formally declared to be zero.
We will give a procedure to compute the above number.

Consider the space 
\begin{align*}
\overline{{M}}_{0,1}(\mathbb{P}^2, d_1) \times \overline{{M}}_{0,1}(\mathbb{P}^2, d_2)\ \longrightarrow\
\mathbb{P}^2_1 \times \mathbb{P}^2_2.
\end{align*}
Let $\mathcal{T}_{d_1, d_2}$ denote the subspace of 
${M}_{0,1}(\mathbb{P}^2, d_1) \times {M}_{0,1}(\mathbb{P}^2, d_2)$, where the degree $d_1$ 
curve and the degree $d_2$ curve evaluate to the same point. Furthermore, they meet each other tangentially at that 
point.
 
We claim that the corresponding cycle determined by the closure is given by
\begin{align}
[\overline{\mathcal{T}}_{d_1, d_2}]\ =\ (\textnormal{ev}_{1}\times \textnormal{ev}_2)^{*}[\Delta_{12}]\cdot 
\Big( c_1(\mathbb{L}_1^{*})+c_1(\mathbb{L}_2^{*}) + c_1(\textnormal{ev}^*(T\mathbb{P}^2_2)) \Big), 
\label{T_cycle}
\end{align}
where $\Delta_{12}$ is the diagonal of $\mathbb{P}^2_1 \times \mathbb{P}^2_2$. 
We will justify the above claim shortly. 
Now note that
\begin{align}
\mathsf{T}^{d_1, d_2}_{m_1, m_2}(n)\ =\ [\mathcal{T}_{d_1, d_2}]\cdot \mathcal{H}_1^{m_1} \cdot \mathcal{H}_2^{m_2} 
\textnormal{ev}_1^{*}(a_1^n). \label{td1d2}
\end{align}
The right hand side of equation \eqref{td1d2} 
can be computed using the intersection of Tautological classes as given in Section \ref{itc_wdvv}.
This gives us the value of $\mathsf{T}^{d_1, d_2}_{m_1, m_2}(n)$. 

Let us now justify equation \eqref{T_cycle}. We denote a point in 
$\overline{M}_{0,1}(\mathbb{P}^2, d_1)$ by $[u_1,y_1]$ 
and a point in 
$\overline{M}_{0,1}(\mathbb{P}^2, d_2)$ by $[u_2,\,y_2]$. Here $u_i$ is the stable 
map and $y_i$ is the marked point. Now note that the inverse image of the diagonal 
via the evaluation map is the set of maps 
$[u_1,y_1]$ and $[u_2,y_2]$ such that $u_1(y_1) = u_2(y_2)$. 
Let us denote the point $u_1(y_1)$ by $p$ (which is same as $u_2(y_2)$). 

Now impose the condition that the image of $u_1$ is tangent to the image of $u_2$ 
at the point $p$. Let us see how we can do that. 
Note that this will happen if the differential of $u_1$ 
(evaluated at $y_1$) vanishes in the normal direction to the $d_2$ curve. More precisely, 
the image of the differential of $u_1$ in the normal direction is a section of the following bundle: 
\[ \mathbb{L}_1^* \otimes \big(\textnormal{ev}^*T\mathbb{P}_2^2/\mathbb{L}_2\big). \]
The Euler class of the above bundle is precisely equal to 
\[\Big( c_1(\mathbb{L}_1^{*})+c_1(\mathbb{L}_2^{*}) + c_1(\textnormal{ev}^*(T\mathbb{P}^2_2)) \Big). \]
This justifies equation \eqref{T_cycle}.

\section{Rational plane quartic with $E_6$ singularity via WDVV}\label{derivation_E6_curves}

In this section, we will explain a new approach, namely the WDVV equation, to obtain the number of rational degree $4$ plane curves with $E_6$ singularity lying in the intersection of $n$ generic lines and passing through $8-n$ general points. Denote this number by $N_4(E_6, n)$. We will explain how we obtain this number $N_4(E_6, n)$ using WDVV. As a result, recursively, we get the following:
\begin{align*}
\begin{split}
N_4(E_6, 0) = &~147\\
N_4(E_6, 1) = &~33\\
N_4(E_6, 2) = &~3 \\
N_4(E_6, n\geq 3) = &~ 0.
\end{split}
\end{align*}
It will be explained how, as before, WDVV equation is used to obtain the above numbers. 
As before, we will do intersection theory on $\overline{M}_{0,4}(\mathbb{P}^2, d)$. 
Let 
$\mathcal{W}_{0,4}(\mathbb{P}^2, d)$ denote the following subset of 
$M_{0,4}(\mathbb{P}^2, d)$: it is 
the space of rational curves (with smooth domains) such that it has a $E_6$ singularity at the first marked point. An element of this space can be visualized by the following picture:
 \begin{figure}[hbt!]
\label{pic idea7}
\begin{center}\includegraphics[scale = 0.8]{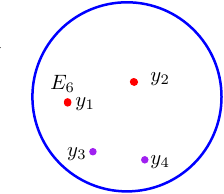}
\end{center}
\end{figure}
\FloatBarrier 
We denote $\overline{\mathcal{W}}_{0,4}(\mathbb{P}^2, d)$ to be its closure, inside $\overline{M}_{0,4}(\mathbb{P}^2, d)$. 
Now consider the forgetful map 
\begin{align*}
\pi: \overline{M}_{0,4}(\mathbb{P}^2, d)\longrightarrow \overline{M}_{0,4}. 
\end{align*}
Define the cycle 
\begin{align*}
\mathcal{\widetilde{Z}}\ :=\ \textnormal{ev}_1^*(a^n) \textnormal{ev}_2^*(a^2)\textnormal{ev}_3^*(a)\textnormal{ev}_4^*(a) \mathcal{H}^{8-n}. 
\end{align*}
This cycle, intersected with 
\[[\pi^{*}(12|34)]\cdot [\overline{\mathcal{W}}_{0,4}(\mathbb{P}^2, d)]\]
in $\overline{M}_{0,4}(\mathbb{P}^2, d)$ will 
give us a number and so does for the intersection $\pi^{*}(13|24)]\cdot [\overline{\mathcal{W}}_{0,4}(\mathbb{P}^2, d)]$ as well. We will explicitly workout these intersections only for $d=4$. While computing the first intersection, we see the $[\overline{\mathcal{W}}_{0,4}(\mathbb{P}^2, d)]\cdot [\pi^*(12|34)]$ will comprise of 
components of degree $d_1, d_2$ (where $d_1+d_2 = 4$) and the mark points $y_1$ and $y_2$ lie on the $d_1$ component, while the 
mark points $y_3$ and $y_4$ lie on the $d_2$ component.

Consider the component of 
$[\overline{\mathcal{W}}_{0,4}(\mathbb{P}^2, d)]\cdot [\pi^*(12|34)]$ that corresponds to the 
breakup $d_1 = 0$ and $d_2 =4$. In this situation, the markings $y_1$ and $y_2$ come together in 
$\mathcal{W}_{0,4}(\mathbb{P}^2, d)$. The resulting object is a degree $4$ curve with $E_6$ singularity, where the singularity lie on the intersection of $n+2$ lines where the constant component is mapped and it is attached with the marked
points $y_3$ and $y_4$ on $d_2$ component. Thus the intersection of 
this component with $\mathcal{\widetilde{Z}}$ will give us 
\begin{align*}
4^2~N_4(E_6, n+2). 
\end{align*}

Next, consider the component of 
$[\overline{\mathcal{W}}_{0,4}(\mathbb{P}^2, d)]\cdot [\pi^*(12|34)]$ that corresponds to the 
breakup $d_1 = 4$ and $d_2 =0$. This happens when $y_3$ and $y_4$ come together in 
$\mathcal{W}_{0,4}(\mathbb{P}^2, d)$. The resulting object is a degree $4$ curve with $E_6$ singularity
with the marked points $y_1$ and $y_2$ on them and a constant component bubble attached 
with the marked points $y_3$ and $y_4$ on that bubble. Thus the intersection of 
this component with $\mathcal{\widetilde{Z}}$ will give us 
\begin{align*}
N_4(E_6, n). 
\end{align*} 

Next, consider the component of $[\overline{\mathcal{W}}_{0,4}(\mathbb{P}^2, d)]\cdot [\pi^*(12|34)]$ that corresponds to the 
breakup $d_1\,>\,0$ and $d_2\,>\,0$. There will be two types of non-trivial situations here, however, when $d_i \,>\, 0$
for all $i \,=\, 1,\, 2$ and the $E_6$ singularity is present in any of the $d_i$ component, then there will be no contribution from
these configurations. This is because $N_d(E_6, n) \,=\, 0$ for all $d \,=\, 1,\,2,\,3$. Let us
describe an essential fact similar to the gluing statement described in the computation of rational curves with cusp.

Define $C_{m_1}^{d_1} \mathsf{T}_{m_2}^{d_2}(n)$ to be the number of rational curves with two components of
degree $d_1$ and $d_2$ such that the $d_1$ component has a cusp and it passes through $m_1$ general
points. Furthermore, the cusp lies in the $d_2$ component, passing through $m_2$ points. The $d_2$ component is tangent to a branch of the cusp at the wedge point. The total configuration 
passes through the intersection of $n$ generic lines satisfying the condition $m_1+m_2+n+1 = 3(d_1+d_2)-3$. Then, the geometric fact that we observed is the following: in the closure of rational curves with a $E_6$ singularity, we have
rational curves having two components with a cusp at one of the components and the cusp lies in the other component such that the other component is tangent to a branch of the cusp. This degeneration can be observed by the following picture:
\begin{figure}[hbt!]
\label{pic idea8}
\begin{center}\includegraphics[scale = 0.6]{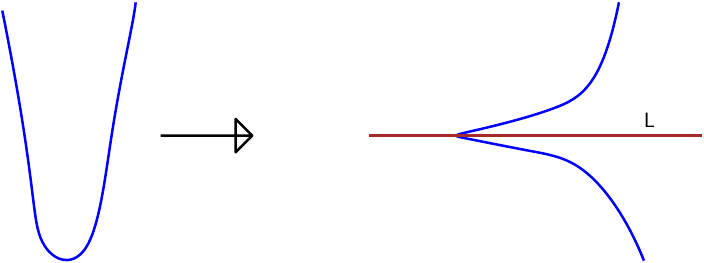}
\end{center}
\end{figure}
\FloatBarrier
The locus of these curves contributes with some multiplicities to the enumeration of rational curves with $E_6$ singularity using the WDVV equation. When the degree of the rational curve is $4$, one can see that there will be only two non-trivial cases for the breakup $(d_1, d_2) = (1, 3)$ and $(d_1, d_2) = (3, 1)$. There will be no contribution from the breaking $(d_1, d_2) = (2, 2)$ since no cuspidal curves of degree $1$ and $2$ exist. Note that the objects in the closure of $\mathcal{W}_{0,4}(\mathbb{P}^2, d)$ can be represented by the following picture:
\begin{figure}[hbt!]
\label{pic idea9}
\begin{center}\includegraphics[scale = 0.7]{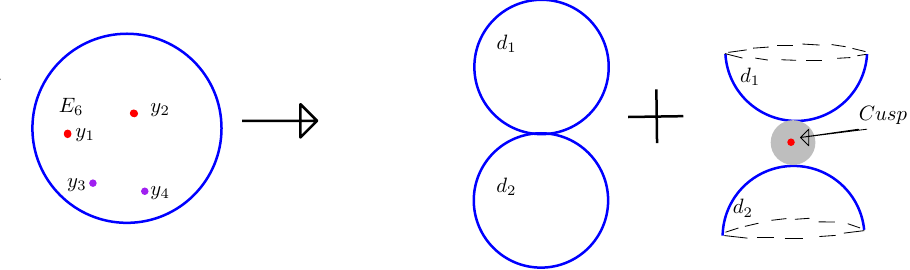}
\end{center}
\end{figure}
\FloatBarrier
The above discussion reflects that for the quartics, there will be no contribution from the configurations of the first term on the right-hand side of the above picture, and the second configuration will contribute to the WDVV equation with certain multiplicities.

The first type of 
non-trivial components will certainly be a degree $d_1 = 1$ curve and a degree $d_2 = 3$ curve such that the rational cubic has a cusp at the nodal point and the point $y_2$ is on the degree $d_1$ component together with the markings $y_3$, $y_4$ are on the degree $d_2$ component, where a branch of the cusp is tangent to the line.

Intersecting this with $\mathcal{\widetilde{Z}}$ gives us 
\begin{align}
\label{E6 wdvv part A1}
\begin{split}
 \binom{6-n}{5-n} \times {C}^{3}_{5-n} \mathsf{T}_{2}^{1}(n) \times 3^2 + \binom{6-n}{6-n} \times {C}^{3}_{6-n} \mathsf{T}_{1}^{1}(n) \times 3^2 
 \end{split}
\end{align}
To see this, choose $5-n$ points out of the available $6-n$ points and 
place the cubic curve (degree $d_2$) through it. This curve will actually be passing through $5-n+1$ points, because 
we are intersecting with $\textnormal{ev}_2^*(a^2)$. We place the line (degree 
$d_1$ curve) through the remaining point and the nodal point. There is a factor of $d_2^2$ 
due to the intersection with $\textnormal{ev}_3^*(a)\textnormal{ev}_4^*(a)$. The crucial fact that we will be using without proof is that this contributes to the WDVV procedure with multiplicity one.

Next, we discuss 
the second type of non-trivial situation.
This comprises a degree $d_1 = 3$ curve having a cusp, and a degree $d_2 = 1$ curve wedged at the ghost bubble. 
The point $y_2$ lies on the degree $d_1$ component. The points $y_3$, $y_4$ lie on the degree $d_2$ 
component and the point $y_1$ lies on the ghost bubble. Furthermore, the line is tangent to a branch of the cusp at the image of $y_1$.

Intersecting this with $\mathcal{\widetilde{Z}}$ gives us 
\begin{align}
\label{E6 wdvv part A2}
 \binom{6-n}{2} \times {C}^{3}_{5-n} \mathsf{T}_{2}^{1}(n) \times 1^2 + \binom{6-n}{1} \times {C}^{3}_{6-n} \mathsf{T}_{1}^{1}(n) \times 1^2 + \binom{6-n}{0} \times {C}^{3}_{7-n} \mathsf{T}_{0}^{1}(n) \times 1^2. 
\end{align} 
One of the most non-trivial fact that we will be using without proof is that this contributes to the WDVV procedure with multiplicity two.\\

Similar to the above, we will now focus on the intersection $[\pi^{*}(13|24)]\cdot [\overline{\mathcal{W}}_{0,4}(\mathbb{P}^2, d)]$ inside $\overline{M}_{0,4}(\mathbb{P}^2, d)$.

Consider the component of 
$[\overline{\mathcal{W}}_{0,4}(\mathbb{P}^2, d)]\cdot [\pi^*(13|24)]$ that corresponds to the
breakup $d_1 \,=\, 0$ and $d_2 \,=\,4$. In this situation, the markings $y_1$ and $y_3$ come together in 
$\mathcal{W}_{0,4}(\mathbb{P}^2, d)$. The resulting object is a degree $4$ curve with $E_6$ singularity, where the singularity lie on the line and the marked points $y_3$ and $y_4$ on $d_2$ component. Thus the intersection of 
this component with $\mathcal{\widetilde{Z}}$ will give us 
\begin{align*}
4~N_4(E_6, n+1). 
\end{align*}

Note the component of 
$[\overline{\mathcal{W}}_{0,4}(\mathbb{P}^2, d)]\cdot [\pi^*(13|24)]$ that corresponds to the 
breakup $d_1 = 4$ and $d_2 = 0$ do not contribute. Also, as pointed out before, there will be no contribution from any
rational curve having two components of degree $d_1$ and $d_2$ with $d_i>0$ for all $i = 1, 2$ and the $E_6$ singularity is present at any of the degree $d_i$ component. Thus, we left with the non-trivial cases as we encountered before.

Similar to the arguments above, we have the contribution from the components where a degree $d_1 = 1$ curve and a degree $d_2 = 3$ curve wedged at the ghost bubble, with the degree $d_2$ 
curve having a cusp at the image of $y_1$, where $y_1$ is in the ghost bubble, and the point $y_2$ is on the degree $d_1$ component together with the markings $y_3$, $y_4$ are on the degree $d_2$ component. Furthermore, the line is tangent to a branch of the cusp.

Intersecting this with $\mathcal{\widetilde{Z}}$ gives us 
\begin{align}
\label{E6 wdvv part B1}
 \binom{6-n}{4-n} \times {C}^{3}_{5-n} \mathsf{T}_{2}^{1}(n) \times 3 \times 1 + \binom{6-n}{5-n} \times {C}^{3}_{6-n} \mathsf{T}_{1}^{1}(n) \times 3 \times 1 + \binom{6-n}{6-n} \times {C}^{3}_{7-n} \mathsf{T}_{0}^{1}(n) \times 3 \times 1.
\end{align}
This contributes with multiplicity one in the WDVV equation.

Next, the component comprises of a degree $d_1= 3$ curve with a cusp, a ghost bubble, and a degree $d_2 = 1$ curve (line) attached to the ghost bubble. 
The point $y_2$ lies on the degree $d_1$ component. The points $y_3$, $y_4$ lie on the degree $d_2$ 
component and the point $y_1$ lies on the ghost bubble. Furthermore, the image of the degree $d_2$ curve is tangent to a branch of the cusp at the image of $y_1$, when intersected with $\mathcal{\widetilde{Z}}$, gives
\begin{align}
\label{E6 wdvv part B2}
 \binom{6-n}{1} \times {C}^{3}_{5-n} \mathsf{T}_{2}^{1}(n)~ \times 1 \times 3 
 + \binom{6-n}{0} \times {C}^{3}_{6-n} \mathsf{T}_{1}^{1}(n)~ \times 1 \times 3 .
\end{align} 
This contributes with multiplicity two in the WDVV equation.\\

Gathering all the above, we get the following contributions:
\begin{align}
\label{LHS of E6 via WDVV}
\begin{split}
[\overline{\mathcal{W}}_{0,4}(\mathbb{P}^2, d)]\cdot [\pi^*(12|34)] \cdot \mathcal{\widetilde{Z}} = 4^2 N_4(E_6, n+2) + N_4(E_6, n) + 9 \binom{6-n}{5-n}~ {C}^{3}_{5-n} \mathsf{T}_{2}^{1}(n) + \\ 9 \binom{6-n}{6-n}~ {C}^{3}_{6-n} \mathsf{T}_{1}^{1}(n) + 
 2 \Big\{ \binom{6-n}{2}~ {C}^{3}_{5-n} \mathsf{T}_{2}^{1}(n) + \binom{6-n}{1}~ {C}^{3}_{6-n} \mathsf{T}_{1}^{1}(n) + \binom{6-n}{0}~ {C}^{3}_{7-n} \mathsf{T}_{0}^{1}(n) \Big\}.
\end{split}
\end{align}
and 
\begin{align}
\label{RHS of E6 via WDVV}
\begin{split}
[\overline{\mathcal{W}}_{0,4}(\mathbb{P}^2, d)]\cdot [\pi^*(13|24)] \cdot \mathcal{\widetilde{Z}}
\, =\, 4N_4(E_6, n+1) + 3 \binom{6-n}{4-n}~ {C}^{3}_{5-n} \mathsf{T}_{2}^{1}(n) + 3 \binom{6-n}{5-n}~ {C}^{3}_{6-n} \mathsf{T}_{1}^{1}(n) + \\ 3 \binom{6-n}{6-n}~ {C}^{3}_{7-n} \mathsf{T}_{0}^{1}(n) + 2 \Big\{ 3 \binom{6-n}{1}~ {C}^{3}_{5-n} \mathsf{T}_{2}^{1}(n) + 3 \binom{6-n}{0}~ {C}^{3}_{6-n} \mathsf{T}_{1}^{1}(n) \Big\}.
\end{split}
\end{align}

Therefore, the WDVV equation reduces the equality of the equations \eqref{LHS of E6 via WDVV} and \eqref{RHS of E6 via WDVV}. Observe that by considering $n = 2$, one can independently apply the above WDVV procedure to directly obtain $N_4(E_6,~2) = 3$. Next, by applying WDVV for the case $n = 1$ and by using the value for $N_4(E_6,~2)$ we will get the number $N_4(E_6,~1)$. Finally, using the values for $N_4(E_6,~2)$ and $N_4(E_6,~1)$ and by applying WDVV for $n= 0$, we obtain the following equality of numbers
\begin{align*}
16 N_4(E_6, 2) + N_4(E_6, 0) + 432 + 2 \times 207 = 4 N_4(E_6, 1) + 621 + 2 \times 144,
\end{align*}
{where the values for} $C_{m_1}^{d_1} \mathsf{T}_{m_2}^{d_2}(n)$ {are taken from the table} appeared in subsection \ref{Table E6}.
From the above, we get the value $N_4(E_6,~0) = N_4(E_6) = 147$.

\subsubsection{Rational cuspidal curves with branched tangency}
\label{Table E6} 

It remains now to compute the numbers 
$C_{m_1}^{d_1} \mathsf{T}_{m_2}^{d_2}(n)$ for $d_2=1$ and $d_1=3$. 
Let us explain how to compute it. 

Note that this number is same as the characteristic 
number of two component curves, 
where the first component is a line, the second component is a cuspidal 
cubic and the line is tangent to the branch of the cuspidal point of the cubic. 
Denote by $\mathcal{D}_1$ and $\mathcal{D}_3$ the space of 
lines and cubics, respectively, in $\mathbb{P}^2$. Note that these are projective spaces of 
dimensions $2$ and $9$ respectively. Define 
\[ \mathsf{M}:= \mathcal{D}_1 \times \mathcal{D}_3 \times \mathbb{P}^2. \] 
The respective hyperplane classes will be denoted by $y_1, y_3$ and $a$. 
In order to compute the desired number, we will do intersection theory on the space 
$\mathsf{M}$. First of all, denote by $\mathsf{A}_2^{\mathsf{F}}$ the 
subspace of elements $(H_1,\, H_3\, p)$ in $\mathsf{M}$ such that the cubic $H_3$ 
has a cusp point at the marked point $p$.
The homology class represented by the closure of this space is given by 
\begin{align*}
[\mathsf{A}_2^{\mathsf{F}}]&\ =\ 24y_3^2 a^2 + 12y_3^3 a + 2y_3^4.
\end{align*}
To see this, we first note that $\mathsf{A}_2^{\mathsf{F}}$ 
is a codimension $4$ cycle; hence it will be given by a degree $4$ polynomial 
in $y_1, y_3$ and $a$. There is no dependence on $y_1$, since it is the pullback of a 
class from $\mathcal{D}_3 \times \mathbb{P}^2$. Next, note that the coefficient of 
$y_3^2 a^2$ is the number of cuspidal cubics in $\mathbb{P}^2$ passing through $7$ 
generic points. Similarly, the coefficient of $y_3^3 a$ is the
number of cuspidal cubics in $\mathbb{P}^2$ passing through $6$ 
generic points, with the cusp lying on a line. Finally, the coefficient of 
$y_3^4$ is the 
number of cuspidal cubics in $\mathbb{P}^2$ passing through $5$ 
generic points, with the cusp lying on a point. These numbers 
are $24$, $12$ and $2$ respectively 
(which have been computed in \cite{R.M}). 

Next, denote by $\mathsf{A}_2^{\mathsf{L}}$ the subspace of
elements $(H_1,\, H_3,\, p)$ in $\mathsf{M}$, such that the cubic has a 
cusp at the point $p$ and the line passes through the point $p$. The 
corresponding homology class is given by 
\begin{align*} 
[\mathsf{A}_2^{\mathsf{L}}]&\ =\ (y_1 +a)\cdot [\mathsf{A}_2^{\mathsf{F}}].
\end{align*}
This is because intersecting with the class $y_1+a$ corresponds to the 
condition of the line passing through 
the point $p$. 

Finally, denote by $\mathsf{P}\mathsf{A}_2$ the subspace of curves, where 
the line is tangent to the branch of the cusp. To see how to determine the homology 
class, first of all, define the incidence variety $\mathrm{I}$ as 
the subspace of
elements $(H_1,\, H_3,\, p)$ in $\mathsf{M}$ such that the point $p$ lies on the line and 
the curve. Assume that $H_1$ is given as the zero set of the linear polynomial $f_1$ 
and $H_3$ is given as the zero set of the cubic polynomial $f_3$. We can think of 
$\mathrm{I}$ as
\begin{align*}
\mathrm{I}&:= \{ ([f_1], [f_3], p)\in \mathsf{M}: f_1(p)=0, ~~f_3(p)=0\}. 
\end{align*}
On top of the incidence variety $\mathrm{I}$, we have a short exact sequence of vector bundles 
\[ 0 \longrightarrow \mathbb{L}:= 
\mathsf{Ker} (\nabla f_1|_p) \longrightarrow T\mathbb{P}^2|_p \longrightarrow 
(\gamma_{\mathcal{D}_1}^* \otimes \gamma_{\mathbb{P}^2}^*)|_p \longrightarrow 0. \]
Define $\lambda:= c_1(\mathbb{L}^*)$. From the above short exact sequence, 
we note that $\lambda\,=\, y_1 - 2a$. 
Now note that we can geometrically think of $\mathsf{PA}_2$ as the zero set of the section 
of a line bundle that is induced by taking the second derivative of $f_3$ along the 
direction of $\mathbb{L}$. This is a section of the bundle 
\[ \gamma_{\mathcal{D}_3}^* \otimes \mathbb{L}^{* 2} \otimes \gamma_{\mathbb{P}^2}^{*3}.\] 
The Euler class of the above line bundle is $(y_3 + 2 \lambda + 3 a)$. Hence, 
\begin{align*}
[\mathsf{P}\mathsf{A}_2]& = [\mathsf{A}_2^{\mathsf{L}}]\cdot (y_3 + 2\lambda + 3 a).
\end{align*}
Our desired number $C_{m_1}^{d_1} \mathsf{T}_{m_2}^{d_2}(n)$ 
is now given by 
\[C_{m_1}^{d_1} \mathsf{T}_{m_2}^{d_2}(n)= [\mathsf{P}\mathsf{A}_2]\cdot y_1^{m_1} y_3^{m_2} a^n. \]
For the convenience of the reader, we will tabulate these numbers as follows:\\
\begin{center}
\begin{tabular}{||c | c | c||} 
 \hline
 When $n = 0$ & When $n = 1$ & When $n = 2$ \\ [1.5ex] 
 \hline
 $(m_1,~m_2) $~\hspace*{1.7cm}~ $C_{m_1}^{3} \mathsf{T}_{m_2}^{1}(0)$ & $(m_1,~m_2) $~\hspace*{1.7cm}~ $C_{m_1}^{3} \mathsf{T}_{m_2}^{1}(1)$ & $(m_1,~m_2) $~\hspace*{1.7cm}~ $C_{m_1}^{3} \mathsf{T}_{m_2}^{1}(2)$ \\[1.2ex]
 \hline\hline
 $(5,~2) $~\hspace*{2cm}~ $5$ & $(4,~2) $~\hspace*{2cm}~ $1$ & $(3,~2) $~\hspace*{2cm}~ $0$ \\ 
 \hline
$(6,~1) $~\hspace*{1.92cm}~ $18$ & $(5,~1) $~\hspace*{2cm}~ $7$ & $(4,~1) $~\hspace*{2cm}~ $1$ \\
 \hline
 $(7,~0) $~\hspace*{1.92cm}~ $24$ & $(6,~0) $~\hspace*{1.92cm}~ $12$ & $(5,~0) $~\hspace*{2cm}~ $2$ \\  
 \hline
\end{tabular}
\end{center}
\vspace*{.3cm}

Note that $C_{m_1}^{d_1} \mathsf{T}_{m_2}^{d_2}(n) = 0$ for all $n \geq 3$. When the degree of the $d_2$ component is higher than $1$, at present, we do not have a formula to compute $C_{m_1}^{d_1} \mathsf{T}_{m_2}^{d_2 \geq 2}(n)$. 

\section{Low degree checks} 
\label{ldc}
All the values that we have computed using our formulas, agree with the values computed earlier by 
Ran (\cite{Ran3}), 
Pandharipande (\cite{Rahul1}), Zinger (\cite{g2p2and3}) and Ernstr\"{o}m and Kennedy (\cite{ken}). 
For the convenience of the reader, we tabulated various low-degree numbers for a few cases. Interested readers are invited to use our Mathematica program (available on request) for several other numbers appearing in various cases.

\subsection{Verification with tangencies}

Using the results in Section \ref{T_num}, we can explicitly compute all possible numbers for $\mathsf{T}^{d_1, d_2}_{m_1, m_2}(n)$. Observe that various possible cases of enumerating rational curves with first order tangencies are incorporated within the symbol $\mathsf{T}^{d_1, d_2}_{m_1, m_2}(n)$. For instance, if we fix the degree $d_1$ component, that is, if we look at the number of two component rational curves of degree $d_1$ and $d_2$ such that the $d_1$ component passes through $3d_1 -1$ 
points, the $d_2$ component passes through $m_2$ points, and the $d_2$ component is tangent to the $d_1$ component at the nodal, where the nodal point 
passes through the intersection of $n$ generic lines, it is straightforward to see that $\mathsf{T}^{d_1, d_2}_{3d_1-1, m_2}(n) = 0$ for $n \geq 2$. In this case, we will tabulate a few low degree numbers as follows:
\begin{center}
\vspace{.2cm}
\begin{tabular}{|c|c|c|c|c|c|c|c|c|} 
\hline
$(d_1, d_2)$ &$(1, 2)$ &$(1, 3)$ & $(1, 4)$ & $(1, 5)$ & $(2, 2)$ & $(2, 3)$ & $(2, 4)$ & $(2, 5)$ \\
\hline 
$(m_1, m_2)$ &$(2, 4) $ &$(2, 7)$ & $(2, 10)$ & $(2, 13)$ & $(5, 4)$ & $(5, 7)$ & $(5, 10)$ & $(5, 13)$ \\
\hline 
$\mathsf{T}^{d_1, d_2}_{m_1, m_2}(0)$ 
& $2$ & $36$ & $2184$ & $335792$ & $6$ & $96$ & $5608$ & $ 846192 $ \\ 
\hline
\end{tabular}
\vspace{.2cm}
\end{center}

Similarly, when $n = 1$, we have
\begin{center}
\vspace{.2cm}
\begin{tabular}{|c|c|c|c|c|c|c|c|c|} 
\hline
$(d_1, d_2)$ &$(1, 2)$ &$(1, 3)$ & $(1, 4)$ & $(1, 5)$ & $(2, 2)$ & $(2, 3)$ & $(2, 4)$ & $(2, 5)$ \\
\hline 
$(m_1, m_2)$ &$(2, 3) $ &$(2, 6)$ & $(2, 9)$ & $(2, 12)$ & $(5, 3)$ & $(5, 6)$ & $(5, 9)$ & $(5, 12)$\\
\hline 
$\mathsf{T}^{d_1, d_2}_{m_1, m_2}(1)$ 
& $1$ & $10$ & $428$ & $51040$ & $2$ & $20$ & $856$ & $ 102080$ \\ 
\hline
\end{tabular}
\vspace{.2cm}
\end{center}
We see that the values obtained by the formula of $\mathsf{T}^{d_1, d_2}_{m_1, m_2}(n)$ is an agreement with those computed by Gathmann's computer program GROWI (implementing the formulas in \cite{Gath1}).

\subsection{Verification with the result by Ernstr\"{o}m and Kennedy}

We now verify our numbers with those computed by Ernstr\"{o}m and Kennedy in \cite[Page 34]{ken}. Notice that the numbers $\mathsf{C}^{d}_{3d-2-n}(n) = 0$ whenever $n\geq 3$. When $n = 0$, a few low degree numbers $\mathsf{C}^{d}_{3d-2}(0)$ are as follows:
\begin{center}
\vspace{.2cm}
\begin{tabular}{|c|c|c|c|c|c|c|c|} 
\hline
$d$ &$3$ &$4$ & $5$ & $6$ & $7$ & $8$ & $9$\\
\hline 
$\mathsf{C}^{d}_{3d-2}(0)$ 
& $24$ & $2304$ & $435168$ & $156153600$ & $97424784000$ &$97958336523264$ & $149437059373232640$\\ 
\hline
\end{tabular}
\vspace{.2cm}
\end{center}
These are in agreement with tables shown in \cite[ArXiv version, Page 27-28]{ken}. Next, when $n = 1$, the numbers $\mathsf{C}^{d}_{3d-3}(1)$ are tabulated
\begin{center}
\vspace{.2cm}
\begin{tabular}{|c|c|c|c|c|c|c|c|} 
\hline
$d$ &$3$ &$4$ & $5$ & $6$ & $7$ & $8$ & $9$\\
\hline 
$\mathsf{C}^{d}_{3d-3}(1)$ 
& $12$ & $864$ & $130896$ & $39223584$ & $21009319488$ &$18506708865792$ & $25119941440608000$\\
\hline
\end{tabular}
\vspace{.2cm}
\end{center}
Similarly, when $n = 2$, the numbers $\mathsf{C}^{d}_{3d-4}(2)$ are as follows:
\begin{center}
\vspace{.2cm}
\begin{tabular}{|c|c|c|c|c|c|c|c|} 
\hline
$d$ &$3$ &$4$ & $5$ & $6$ & $7$ & $8$ & $9$\\
\hline 
$\mathsf{C}^{d}_{3d-4}(2)$ 
& $2$ & $102$ & $12024$ & $2953656$ & $1341437280$ &$1026019929312$ & $1230836838698880$\\
\hline
\end{tabular}
\vspace{.2cm}
\end{center}
These are all in agreement with the tables \cite[ArXiv version, Page 25-26]{ken}, as expected.

\section{Future directions} 

Some generalizations we hope to think in the future (by extending the idea of this manuscript) 
are as follows: 
\begin{enumerate}
 \item\label{i1} Extending this to del Pezzo surfaces and get the numbers obtained in \cite{BVSRM}. 
 \item\label{i2} Extending this to planar curves in $\mathbb{P}^3$ and recover the numbers obtained in \cite{Rahul_Rit}. 
 \item\label{i3} Enumerating rational cuspidal curves in $\mathbb{P}^n$ (not necessarily planar) and recover the numbers computed in 
 \cite{g0pr}.
\item \label{i4} Enumerating rational curves in $\mathbb{P}^2$ with a singularity which in parametric form is given by 
 $t \longrightarrow (t^2, t^5)$. This if done would be a completely new result.
\item \label{i5} Enumerating curves rational curves in $\mathbb{P}^2$ with two cusps. 
\end{enumerate}
Problems \ref{i1} and \ref{i2} should be doable without too much further effort; the main obstacle there 
is likely to be of a computational nature. We are not sure as yet how tractable problem \ref{i3}. 
We need to see if the one can prove a similar gluing Theorem for $\mathbb{P}^n$ when $n \geq 3$. The 
{Problem \ref{i4}} would both require much more geometrical input. We would need to 
prove more refined gluing theorems, to identify the cycle when we pull back the WDVV equation from 
$\overline{M}_{0,4}$. We intend to investigate this in the future. Problem \ref{i5} is also 
going to require 
{an} additional geometric input. We will need to see what happens when two cusps 
come together. 

\section{Acknowledgement} 
We thank the referees for their helpful comments.

\bibliography{Degreed} 
\bibliographystyle{plain}

\end{document}